\documentclass[11pt]{amsart}
\usepackage[parfill]{parskip}    

\usepackage[hidelinks]{hyperref}
\usepackage{url}
\usepackage{amsmath}
\usepackage{graphicx}
\usepackage{amssymb}
\usepackage{mathtools}
\usepackage[T1]{fontenc}
\usepackage{tikz}
\usepackage{enumerate}

\newcommand*{\rom}[1]{\expandafter\@slowromancap\romannumeral #1@}
\newcommand{\RNum}[1]{\uppercase\expandafter{\romannumeral #1\relax}}
\newcommand{\Z}{\mathbb{Z}}  
\newcommand{\R}{\mathbb{R}}  
\newcommand{\C}{\mathbb{C}}  
\newcommand{\Q}{\mathbb{Q}}
\newcommand{\acts}{\curvearrowright}

\newcommand{\Rat}{\mathrm{Rat}}
\newcommand{\Eig}{\mathrm{Eig}}

\newcommand{\ord}{\mathrm{ord}}
\newcommand{\SL}{\mathrm{SL}}
\DeclarePairedDelimiter\ceil{\lceil}{\rceil}
\DeclarePairedDelimiter\floor{\lfloor}{\rfloor}
\newcommand\labs[1]{\left\lvert#1\right\rvert} 

\newcommand\lnorm[1]{\left\lVert#1\right\rVert}

\newcommand{\T}{\mathbb{T}}

\renewcommand{\leq}{\leqslant}
\renewcommand{\geq}{\geqslant}

\theoremstyle{plain} 
\newtheorem{thm}{Theorem}[section]
\newtheorem{lem}[thm]{Lemma}
\newtheorem{pro}[thm]{Proposition}

\newtheorem{thmAlph}{Theorem}[section]

\newtheorem{corAlph}{Corollary}[section]

\theoremstyle{definition} 
\newtheorem{exam}[thm]{Example}

\theoremstyle{remark} 
\newtheorem{rema}[thm]{Remark}

\title[Primitive Averages and Directional Expansivity]{Primitive Averages, Directional Expansivity,\\ and Quantitative Twisted Recurrence for Ergodic $\Z^d$-Actions}
\author{Rickard Cullman, Sean Skinner}

\address{Department of Mathematics, Chalmers and University of Gothenberg, Sweden}
\email{cullman@chalmers.se}
\address{School of Mathematics and Statistics, University of Sydney, Australia}
\curraddr{}
\email{sean.skinner@sydney.edu.au}
\thanks{}

\date{\today}                                           

\begin{document}
\begin{abstract}
We prove two new results about probability preserving actions $T:\Z^d \acts (X,\mu)$. First, for a function $f\in L^2(\mu)$, we provide an explicit formula for the $L^2(\mu)$-limit of the average
\[\frac{1}{\labs{Q_N^\mathcal{P}}}\sum_{v \in Q_N^\mathcal{P}} T_v f\]
where $\mathcal{P}\subset \Z^d$ is the set of primitive vectors, i.e. those for which the greatest common divisor of its components is $1$, and $Q_N^\mathcal{P}= [-N,N]^d\cap \mathcal{P}$.
Second, for a set $A\subset X$ with $\mu(A)>0$, we provide a spectral condition under which the set of $\varepsilon$-expansive directions
\[\left\{ v\in \Z^d \, : \, \mu\left(\bigcup_{n\in \Z} T_{nv}A\right)>1-\varepsilon\right\} \]
has lower density very close to $1$. As an application of our techniques we are also able to prove a quantitative variant of a twisted multiple recurrence theorem of Björklund, Fish and the first author \cite{BjCF}.
\end{abstract}
\maketitle

\section{Introduction}
Fix  an integer $d\geq 2$. Let $T:\Z^d \acts (X,\mu)$ be a probability preserving system\footnote{We choose not to include the underlying $\sigma$-algebra in our notation and moving forward all subsets of a measurable space will be assumed to be measurable. We will also always assume that $L^2(\mu)$ is separable.} and denote $Q_N=[-N,N]^d\cap \Z^d$.
For any $f\in L^2(\mu)$, von Neumann's mean ergodic theorem implies that
\begin{equation}\label{eq: MET}
\lim_{N\to \infty} \frac{1}{\labs{Q_N}}\sum_{v \in Q_N} T_v f = P_{\mathcal{I}}f
\end{equation}
in $L^2(\mu)$ where $P_\mathcal{I}$ denotes the $L^2(\mu)$-projection onto the subspace of $T$-invariant functions
\[
\mathcal{I}= \{f\in L^2(\mu) \, : \, T_v f = f \text{ for all }v\in \Z^d\}.
\]
Our first result gives an explicit analogue of the formula in equation \eqref{eq: MET} when the averages are taken only over the primitive vectors $\mathcal{P}\subset \Z^d$, where a non-zero vector $v\in \Z^d$ is called primitive if the greatest common divisor of its non-zero components is equal to $1$.

Unlike equation \eqref{eq: MET}, the limiting value for the primitive averages also depends on how $f$ correlates with eigenfunctions of finite order, and the strength of this dependence is determined by the arithmetic distribution of primitive lattice points in residue classes. We denote $Q_N^\mathcal{P}:=Q_N\cap \mathcal{P}$.


\begin{thmAlph}\label{Th: MET for primitives}
Let $T:\Z^d \acts (X,\mu)$ be a probability preserving action and let $f\in L^2(\mu)$. Then
\begin{equation}\label{eq: Formula for primitive MET lim}
\lim_{N\to \infty} \frac{1}{\labs{Q_N^\mathcal{P}}} \sum_{v \in Q_N^\mathcal{P}} T_v f 
= P_{\mathcal{I}}f + \sum_{q\geq 2} \frac{\boldsymbol{\mu}(q)}{J_d(q)} P_{q}f,
\end{equation}
where $\boldsymbol{\mu}: \Z_{>0}\to \R$ is the M\"{o}bius function\footnote{We will always use the bold symbol $\boldsymbol{\mu}$ for the M\"{o}bius function, reserving the un-bolded $\mu$ for a probability measure.}, $J_d: \Z_{>0}\to \R$ is the Jordan totient function
\[
J_d(q): = \labs{\left\{ a=(a_1,\ldots,a_d)\in (\Z/q\Z)^d\, : \, \gcd(a,q) = 1\right\}},
\]
and $P_q$ is the orthogonal projection onto the $L^2(\mu)$ subspace generated by eigenfunctions\footnote{For the precise definition of eigenfunctions and their order see Section \ref{Sec: Background}.} of order $q$.
\end{thmAlph}
In particular, the limit always exists and depends only on the projection of $f$ to the rational Kronecker factor, i.e. the closed subspace generated by eigenfunctions of finite order. Since $J_d(q)\to \infty$ as $q\to \infty$, the contribution of the order-$q$ eigenfunctions in \eqref{eq: Formula for primitive MET lim} becomes negligible for large $q$.

If the system is totally ergodic, i.e. if $\Z^d$ and all of its finite index subgroups act ergodically, then the formula collapses to the expected constant limit. For $E\subset \Z^d$, we define the lower densities of $E$ with respect to $Q_N$ and $Q_N^\mathcal{P}$ by
\[
\underline{d}_{Q_N}\left(E \right):= \liminf_{N\to \infty}\frac{\labs{E\cap Q_N}}{\labs{Q_N}} 
\qquad\text{and}\qquad 
\underline{d}_{Q_N^\mathcal{P}}\left(E \right):= \liminf_{N\to \infty}\frac{\labs{E\cap Q_N^\mathcal{P}}}{\labs{Q_N^\mathcal{P}}}
\]
respectively.

\begin{corAlph}\label{Cor: Primitive return times}
Let $T:\Z^d \acts (X,\mu)$ be a totally ergodic system and let $f\in L^2(X)$. Then
\[
\lim_{N\to \infty} \frac{1}{\labs{Q_N^\mathcal{P}}} \sum_{v \in Q_N^\mathcal{P}} T_v f = \int_X f \,d\mu.
\]
Moreover, if $A\subset X$ has $\mu(A)>0$, then for every $\varepsilon>0$ we have that
\[
\underline{d}_{Q_N^\mathcal{P}} \left( \left\{v \in \mathcal{P} \, : \, \mu(A\cap T_v A) > \mu(A)^2 - \varepsilon \right\} \right)>0.
\]
\end{corAlph}

Our second result concerns cyclic subgroups of $\Z^d$ along which the orbit of a positive-measure set almost covers the whole space. Given $A\subset X$ with $\mu(A)>0$, we say that a vector $v\in \Z^d$ is an \textit{$\varepsilon$-expansive direction} for $A$ if
\begin{equation}
\mu\left(\bigcup_{n\in \Z}T_{nv} A\right)>1-\varepsilon.
\end{equation}
Expansive directions were first studied by Björklund and Fish \cite{BjF} where they played a central role in their study of simplicies in large subsets of $\Z^d$, and have since found further applications in \cite{FS} and \cite{BjCF}.

The existence of expansive directions depends on how the spectral measure of $A$ is distributed on points of finite order. Recall that the spectral measure $\sigma_A$ of $A$ with respect to $T: \Z^d \acts (X,\mu)$ is the unique finite Borel measure on $\T^d$ satisfying that
\begin{equation}\label{eq: Def of spec measure}
\mu(A\cap T^{v}A)= \int_{\T^d} e(v\cdot\alpha)\, d\sigma_A(\alpha) 
\quad \text{for every }v\in \Z^d.
\end{equation}
The order $\ord(\alpha)$ of a point $\alpha \in \T^d$ is defined to be the smallest positive integer $n$ for which $n\alpha =0$ if such an integer exists, and $\infty$ otherwise. In \cite{BjF} Björklund and Fish proved in that if
\[
\sigma_A(\{\alpha \in \T^d \, : \, 1<\ord(\alpha)<\infty\})<\varepsilon \mu(A)^2,
\]
then the set of $\varepsilon$-expansive directions for $A$ is non-empty\footnote{In fact, they showed that under the same assumption, $\varepsilon$-expansive directions for $A$ exist in any \textit{haystack} $H\subset \Z^d$, where a haystack is an infinite subset $H\subset \Z^d$ such that every distinct $v_1,\ldots, v_d \in H$ are linearly independent. See \cite[Theorem 3.1]{BjF}.}.
Our second theorem strengthens this conclusion from non-emptiness to large lower density under a weaker spectral assumption. Indeed, instead of requiring that $\sigma_A$ gives small mass to the infinite set of all points with order less than $\infty$, we only need to control the mass on the finitely many points of order at most $M$ for some constant $M>0$.

\begin{thmAlph}\label{Th: Size of expansive vectors}
For every $\delta,\varepsilon,\eta>0$, there exist 
\[
M=M(\delta,\varepsilon,\eta)>0
\qquad\text{and}\qquad
\kappa = \kappa(\delta,\varepsilon,\eta)>0
\]
such that the following is true. For any probability preserving action $T:\Z^d \acts (X,\mu)$ and any $A\subset X$ with $\mu(A)\geq \delta$, if
\[\sigma_A(\{\alpha \in \T^d\, : \, 1< \ord(\alpha) \leq M\})<\kappa,\]
then
\[
\underline{d}_{Q_N}\left(\left\{ v\in \Z^d \, : \, \mu\left(\bigcup_{n\in \Z} T_{nv}A\right)>1-\varepsilon\right\}\right)> 1-\eta.
\] 
\end{thmAlph}

Not every set in a probability preserving system satisfies the spectral assumption in Theorem \ref{Th: Size of expansive vectors}. However, when $T$ acts ergodically, the ergodic measure increment argument from \cite{FS} allows one to pass to a bounded finite-index subaction and an ergodic component for which the assumption is satisfied, and so we have the following corollary to Theorem \ref{Th: Size of expansive vectors}.

\begin{corAlph}\label{Cor: Expansive vectors}
For every $\delta,\varepsilon,\eta>0$, there exists some positive integer $k_0 = k_0(\delta,\varepsilon,\eta)$ such that the following is true. For every ergodic action $T: \Z^d \acts (X,\mu)$ and every $A\subset X$ with $\mu(A)\geq \delta$, there exists an integer $1\leq k \leq k_0$ and an ergodic component $\nu$ of $\mu$ with respect to the sub-action of $k\Z^d$ such that $\nu(A)\geq \mu(A)$ and
\[
\underline{d}_{Q_N}\left(\left\{ v\in \Z^d \, : \, \nu\left(\bigcup_{n\in \Z} T_{nv}A\right)>1-\varepsilon\right\} \right) > 1-\eta.
\] 
\end{corAlph}

Using Theorems \ref{Th: MET for primitives}\footnote{Actually the proof we present does not use Theorem \ref{Th: MET for primitives} per se, as it turns out to be more direct to use the main technical input to Theorem \ref{Th: MET for primitives} instead, namely Proposition \ref{Prop: Formula for exp sums over primitives}.} and \ref{Th: Size of expansive vectors} we are also able to prove a quantitative variant of a twisted multiple recurrence theorem of Björklund, Fish and the first author. We first recall their theorem.

\begin{thm}[{\cite[Theorem 1.6]{BjCF}}] \label{Theorem: Non-quantitative multiple twisted recurrence}
For any ergodic probability preserving action $T:\Z^d \acts (X,\mu)$ and any $A\subset X$ with $\mu(A)>0$, there exists a positive integer $k=k(A)$ such that for any $d$-linearly independent vectors $v_1, \ldots, v_d\in \Z^d$, there exists some $\gamma \in \SL_d(\Z)$ with
\begin{equation*}
\mu\left( A\cap T_{\gamma k v_1}A\cap \ldots \cap T_{\gamma k v_d}A\right)>0.
\end{equation*}
\end{thm}

Theorem \ref{Theorem: Non-quantitative multiple twisted recurrence} is non-quantitative in the sense that the integer $k$ may depend on the set $A$, rather than only on $\mu(A)$. One might hope to strengthen it by ensuring that $k\leq k_0$ for some $k_0$ depending only on $\mu(A)$, however it turns out that this is not possible as the following example shows.

\begin{exam}\label{Example: Counter example}
For every positive integer $k_0$, there exists an ergodic system $T: \Z^d \acts (X,\mu)$ and a set $A\subset X$ with $\mu(A)=2^{-d}$ such that for every $k=1,\ldots,k_0$, there exists a vector $v_k\in \Z^d$ for which
\[
\mu(A\cap T_{\gamma kv_k}A)=0 
\quad \text{for every }\gamma \in \SL_{d}(\Z).
\]
\end{exam}

In light of Example \ref{Example: Counter example}, any quantitative variant of Theorem \ref{Theorem: Non-quantitative multiple twisted recurrence} must require some further restriction on the configurations $v_1,\ldots,v_d$ for which the conclusion is satisfied. Indeed, restricting ourselves to only those configurations for which $\labs{\det(v_1,\ldots,v_d)}$ is bounded, we prove the following quantitative variant of Theorem \ref{Theorem: Non-quantitative multiple twisted recurrence}.

\begin{thmAlph}\label{Theorem: Quantitative multiple twisted recurrence}
For any $D\in \Z_{>0}$ and any $\delta>0$, there exists a positive integer $k_0=k_0(D,\delta)$ such that for every ergodic action $T:\Z^d \acts (X,\mu)$ and every $A\subset X$ with $\mu(A)\geq \delta$, there exists an integer $1\leq k\leq k_0$ such that the following holds. For any $v_1,\ldots,v_d \in \Z^d$ with 
\[
0< \labs{\det(v_1,\ldots,v_d)}\leq D,
\]
there exists some $\gamma \in \SL_d(\Z)$ with
\[
\mu\left( A\cap T_{\gamma k v_1}A\cap \ldots \cap T_{\gamma k v_d}A\right)>0.
\]
\end{thmAlph}
Unlike the proof of Theorem \ref{Theorem: Non-quantitative multiple twisted recurrence} in \cite{BjCF}, our proof of Theorem \ref{Theorem: Quantitative multiple twisted recurrence} does not use random walk theory, and in particular does not rely on the deep equidistribution results of \cite{BFLM}. Our methods could also be used to give a new proof Theorem \ref{Theorem: Non-quantitative multiple twisted recurrence} which would remove the reliance on random walks entirely, but we have decided not to pursue that direction here, leaving the details to the motivated reader.

By a routine application of Furstenberg's correspondence principle, Theorem \ref{Theorem: Quantitative multiple twisted recurrence} has the following combinatorial consequence. Recall that the upper Banach density of $E\subset \Z^d$ is
\[
d^*(E):=\limsup_{N\to \infty} \sup_{v \in \Z^d} 
\frac{\labs{E\cap (v+ [0,N-1]^d)}}{N^d}.
\]

\begin{corAlph}
For any $D\in \Z_{>0}$ and any $\delta>0$, there exists a positive integer $k_0=k_0(D,\delta)$ such that for every $E\subset \Z^d$ with $d^*(E)\geq \delta$, there exists an integer $1\leq k\leq k_0$ such that for any $v_1,\ldots,v_d \in \Z^d$ with
\[
0< \labs{\det(v_1,\ldots,v_d)}\leq D,
\]
there exists some $v_0\in E$ and $\gamma \in \SL_d(\Z)$ with
\[
v_0 + k\gamma v_i \in E 
\qquad \text{for all }i=1,\ldots,d.
\]
\end{corAlph}

\medskip
\noindent\textit{Asymptotic notation.} 
All asymptotic notation is taken as $N\to\infty$. We write $o_N(1)$ for any quantity which tends to $0$ as $N\to\infty$, and write $a_N=O(b_N)$ if there is a constant $C>0$, independent of $N$, such that $|a_N|\leq C|b_N|$ for all sufficiently large $N$.

\medskip
\noindent\textit{Acknowledgments.}
We are grateful to Michael Björklund and Alexander Fish for their guidance and encouragement. S.S. is particularly thankful to Michael Björklund and Chalmers University for their hospitality in June 2025, when this work began. S.S. was supported by the Australian Research Council through grant DP240100472.

\section{Background}\label{Sec: Background}
We begin by recalling some relevant background material. Let $d$ be a positive integer and let $T: \Z^d \acts (X,\mu)$ be a probability preserving action.

Given any $f\subset L^2(\mu)$ the spectral measure $\sigma_f$ of $f$ with respect to $T:\Z^d \acts (X,\mu)$ is the unique finite Borel measure on $\T^d$ with 
\begin{equation}\label{eq: Def of spec measure}
\int_X f \cdot T_v \overline{f}\, d\mu= \int_{\T^d} e(v\cdot\alpha)\, d\sigma_f(\alpha) \quad \text{for every }v\in \Z^d,
\end{equation}
where $e(x):=\exp(2\pi i x)$ and $\cdot$ is the standard dot product. If $A\subset X$ then we write $\sigma_A$ for $\sigma_{\mathbf{1}_A}$. We always have that $\sigma_A(\{0\})\geq \mu(A)^2$, with equality in the case that $T$ acts ergodically\footnote{See for instance \cite[Lemma 2.5]{BjF}.}.

The order $\ord(\alpha)$ of a point $\alpha \in \T^d$ is defined to be the smallest positive integer $n$ such that $n\alpha = 0$ in $\T^d$ if such an integer exists, and $\infty$ otherwise. For $M>0$ we denote
\[
\mathrm{Rat}(M): = \{\alpha \in \T^d\, : \, 1< \ord(\alpha) \leq M\}.
\]

For any $\alpha \in \T^d$, a non-zero function $f\in L^2(\mu)$ is called an \textit{$\alpha$-eigenfunction} or an \textit{eigenfunciton with eigenvalue $\alpha$} if
\[ T_v f = e(v\cdot \alpha) f \qquad\text{for all }v\in \Z^d\]
and we denote the closed $L^2(\mu)$ subspace spanned by all $\alpha$-eigenfunctions by $\Eig_T(\alpha)$. The order of an $\alpha$-eigenfunction is defined to be $\ord(\alpha)$ and for each positive integer $q$ we write
\[ \mathcal{K}_q := \bigoplus_{\substack{\alpha \in \T^d \\ \ord(\alpha)=q}} \Eig_T(\alpha)\]
for the subspace generated by eigenfunctions of order $q$. The \textit{rational Kronecker factor} is then defined to be
\[ \mathcal{K}_\text{rat}: = \overline{\bigoplus_{q\geq 1} \mathcal{K}_q}.\]

We will also make use of the fact that $\mathcal{P}$ has positive density in $\Z^d$ with respect to $Q_N$, a proof of which can be found in \cite{N}. More precisely,
\begin{equation}\label{eq: Primitive density exists}
\frac{\labs{Q_N^\mathcal{P}}}{\labs{Q_N}} = \frac{1}{\zeta(d)} + o_N(1)
\end{equation}
where $\zeta$ is the Riemann zeta function. For more details, see Section \ref{Section: Exp sums}, where in particular we also provide a proof of equation \eqref{eq: Primitive density exists}.
\section{The case of $d=1$ in Theorem \ref{Theorem: Quantitative multiple twisted recurrence}}
As the notions of primitive vectors and expansive directions are slightly degenerate in dimension $1$, we have only stated our results in the introduction for $d\geq 2$. We remark however that the analogous version of Theorem \ref{Theorem: Quantitative multiple twisted recurrence} for $d=1$ is still true, and in fact follows easily from Poincar\'{e} recurrence. In particular no ergodicity assumption is required. Indeed, the dimension $d=1$ analogue of Theorem \ref{Theorem: Quantitative multiple twisted recurrence} states that for any positive integer $D$ and any $\delta>0$, there exists a positive integer $k_0 = k_0(D,\delta)$ such that for any probability preserving action $T: \Z \acts (X,\mu)$ and any $A\subset X$ with $\mu(A)\geq \delta$, there exists some $1\leq k\leq k_0$ such that
\[\mu(A\cap T_{km}A)>0 \quad \text{for all } m=1,\ldots,D.\]
Applying Poincar\'{e} recurrence to the $\Z$ action of $T\times T^2 \times \ldots \times T^D$ on $(X^D,\mu^{\otimes^D})$ and the set $B = A\times\ldots \times A\subset X^D$ yields the desired result as the first non-trivial return time can always be bounded in terms of $(\mu^{\otimes^D}(B))^{-1}$, which of course is at most $\delta^{-D}$.

For the remainder of the paper let us fix a dimension $d\geq 2$, with the understanding that all constants, explicit and implied, will in general depend on the dimension $d$.
\section{A mean ergodic theorem for primitive vectors}
In this section we prove Theorem \ref{Th: MET for primitives} using the following exponential-sum formula as a black box, and we use it to deduce Corollary \ref{Cor: Primitive return times}.
\begin{pro}\label{Prop: Formula for exp sums over primitives}
For any $\alpha \in \T^d$ we have that
\[g(\alpha): = \lim_{N\to \infty}\frac{1}{\labs{Q_N^\mathcal{P}}} \sum_{v\in Q_N^{\mathcal{P}}} e(v\cdot \alpha)=
\begin{cases}
1 & \text{ if } \alpha=0\\
\frac{1}{J_d(q)}\boldsymbol{\mu}(q) & \text{ if } \ord(\alpha)=q\geq 2\\
0 & \text{ if }\ord(\alpha)=\infty\\
\end{cases}\]
where $\boldsymbol{\mu}: \Z_{>0}\to \R$ is the M\"{o}bius function and $J_d: \Z_{>0}\to \R$ is the Jordan totient function
\[
J_d(q) = \labs{\left\{ a=(a_1,\ldots,a_d)\in (\Z/q\Z)^d\, : \, \gcd(a,q) = 1\right\}}.
\]
\end{pro}
The proof of Proposition \ref{Prop: Formula for exp sums over primitives} is purely number-theoretic and is postponed to Section \ref{Section: Exp sums}, so that the ergodic-theoretic details can be presented without interruption.
\begin{proof}[Proof of Theorem \ref{Th: MET for primitives} via Proposition \ref{Prop: Formula for exp sums over primitives}]
Fix some probability preserving action $T:\Z^d \acts (X,\mu)$ and for any positive integers $N$ and $q$ denote
\[ A_N:= \frac{1}{\labs{Q_N^\mathcal{P}}} \sum_{v\in Q_N^\mathcal{P}} T_v \qquad \text{and} \qquad c_q:=\frac{\boldsymbol{\mu}(q)}{J_d(q)}.\]
Denote by $L: L^2(\mu) \to L^2(\mu)$ the proposed limit operator
\[ Lf  := \sum_{q=1}^\infty c_q P_q f \qquad \text{for any } f\in L^2(\mu)\]
where $P_q$ is orthogonal projection onto $\mathcal{K}_q$.
It is easy to see that $L$ is a linear operator, and in fact it is an $L^2(\mu)$ contraction since by pairwise orthogonality of each $\mathcal{K}_q$ and the fact that each $\labs{c_q}\leq 1$, for any $f\in L^2(\mu)$ we can estimate
\[ \lnorm{Lf}_2^2 = \sum_{q=1}^\infty \lnorm{c_q P_q f}_2^2 \leq \sum_{q=1}^\infty \lnorm{P_q f}_2^2 \leq \lnorm{f}_2^2.\]
We must show that
\begin{equation}\label{eq: AN f goes to Lf}
\lim_{N\to \infty} A_{N}f = L f \qquad \text{for any } f\in L^2(\mu).
\end{equation}
First consider the case that $h\in L^2(\mu)$ is orthogonal to $\mathcal{K}_\text{rat}$. Then for each $\alpha \in \T^d$, equation \eqref{eq: Def of spec measure} together with the mean ergodic theorem applied to the unitary action of $\overline{e(v\cdot \alpha)}T_v$ implies\footnote{See \cite[Lemma 5.2]{FS}.} that  $\lnorm{P_{\Eig_T(\alpha)}h}_2^2 = \sigma_h(\{\alpha\})$ where $P_{\Eig_T(\alpha)}$ is the orthogonal projection onto $\Eig_T(\alpha)$. As $h\perp \mathcal{K}_\text{rat}$ then we must have that $\sigma_h(\Q^d/\Z^d)=0$. We can then calculate
\begin{align*}
\lnorm{A_N h}_2^2 = \lnorm{\frac{1}{\labs{Q_N^\mathcal{P}}} \sum_{v\in Q_N^{\mathcal{P}}} T_v h}_2^2 &= \left \langle \frac{1}{\labs{Q_N^\mathcal{P}}} \sum_{v\in Q_N^{\mathcal{P}}} T_v h ,\frac{1}{\labs{Q_N^\mathcal{P}}} \sum_{w\in Q_N^{\mathcal{P}}} T_w h \right\rangle\\
&=\frac{1}{\labs{Q_N^\mathcal{P}}^2} \sum_{v,w\in Q_N^{\mathcal{P}}} \left \langle T_{v-w} h, h\right \rangle \\
&= \frac{1}{\labs{Q_N^\mathcal{P}}^2} \sum_{v,w\in Q_N^{\mathcal{P}}} \int_{\T^d} e((v-w)\cdot \alpha) \, d\sigma_h(\alpha)\\
&=  \int_{\T^d}  \labs{\frac{1}{\labs{Q_N^\mathcal{P}}}\sum_{v\in Q_N^{\mathcal{P}}} e(v\cdot \alpha)}^2 d\sigma_h(\alpha).
\end{align*}
By Proposition \ref{Prop: Formula for exp sums over primitives} and the dominated convergence theorem then
\begin{equation}\label{eq: prim avg kills Krat perp}
\lim_{N\to \infty} \lnorm{A_N h}_2^2  = \int_{\T^d} \labs{g(\alpha)}^2 d\sigma_h(\alpha) = 0
\end{equation}  
where in the final equality we use that $g(\alpha)=0$ for $\T^d$ with $\ord(\alpha)=\infty$ and that $\sigma_h(\Q^d/\Z^d)=0$. Clearly $Lh = 0$, and so we have that equation \eqref{eq: AN f goes to Lf} holds on $\mathcal{K}_{\text{rat}}$.

Now let $\alpha \in \T^d$ have $\ord(\alpha)=q<\infty$ and let $f_\alpha \in \Eig_T(\alpha)$. As $\Eig_T(\alpha)\subset \mathcal{K}_q$ then $L f_\alpha = c_q f_\alpha$ so we can calculate
\begin{align*}
\lnorm{A_N f_{\alpha} - L f_{\alpha}}_2^2 &= \lnorm{\frac{1}{\labs{Q_N^\mathcal{P}}} \sum_{v\in Q_N^{\mathcal{P}}} T_v f_\alpha - c_q f_\alpha}^2_2\\
&= \lnorm{\left(\frac{1}{\labs{Q_N^\mathcal{P}}}\sum_{v\in Q_N^{\mathcal{P}}} e(v\cdot \alpha)-\frac{\boldsymbol{\mu}(q)}{J_d(q)}\right)f_\alpha}_2^2\\
&\leq \labs{\frac{1}{\labs{Q_N^\mathcal{P}}}\sum_{v\in Q_N^{\mathcal{P}}} e(v\cdot \alpha)-\frac{\boldsymbol{\mu}(q)}{J_d(q)}}^2 \lnorm{f_\alpha}_2^2
\end{align*}
which goes to $0$ as $N\to \infty$ by Proposition \ref{Prop: Formula for exp sums over primitives} and so equation \eqref{eq: AN f goes to Lf} also holds for any eigenfunction of finite order. By linearity of $L$, the desired formula also holds for finite linear combinations of finite order eigenfunctions. Since finite linear combinations of finite order eigenfunctions are dense in $\mathcal{K}_\text{rat}$, then by a standard density argument using that $A_N$ and $L$ are both $L^2(\mu)$ contractions we conclude that equation \eqref{eq: AN f goes to Lf} also holds for any $f\in \mathcal{K}_{\text{rat}}$. Since $L^2(\mu) = \mathcal{K}_{\text{rat}} \oplus \mathcal{K}_{\text{rat}}^\perp$ then we are done.
\end{proof}
\begin{proof}[Proof of Corollary \ref{Cor: Primitive return times}]
Let $T:\Z^d \acts (X,\mu)$ be totally ergodic. It is easy to see that the only finite order eigenfunctions are constant a.e., and so for any $f\in L^2(\mu)$ we have that $P_q f = 0$ for all $q\geq 2$. By Theorem \ref{Th: MET for primitives} and ergodicity of $T$ we then have that
\[ \lim_{N\to \infty}\frac{1}{\labs{Q_N^\mathcal{P}}}\sum_{v\in Q_N^\mathcal{P}} T_v f = P_{\mathcal{I}} f = \int_X f \, d\mu.\] 
Now suppose $f =\mathbf{1}_A$ for some $A\subset X$ with $\mu(A)>0$. Then by continuity of the inner product and the previous equation we have that
\begin{align}
\lim_{N\to \infty}\frac{1}{\labs{Q_N^\mathcal{P}}}\sum_{v\in Q_N^\mathcal{P}} \mu(A\cap T_v A)& =   \left\langle \lim_{N\to \infty}\frac{1}{\labs{Q_N^\mathcal{P}}}\sum_{v\in Q_N^\mathcal{P}}T_v \mathbf{1}_A,\mathbf{1}_A\right\rangle \nonumber \\
&= \mu(A)^2. \label{eq: avg of A cap Tv A for TE}
\end{align}
For any $\varepsilon>0$ we must then have that
\[ \underline{d}_{Q_N^\mathcal{P}} (\{ v\in \mathcal{P} \, : \, \mu(A\cap T_v A)>\mu(A)^2 - \varepsilon\})>0\]
since otherwise we have a contradiction to equation \eqref{eq: avg of A cap Tv A for TE}.
\end{proof}
\section{Counting expansive vectors}
In this section we prove Theorem \ref{Th: Size of expansive vectors}.
\begin{lem}\label{Lemma: Kernels and densities}
Let $\alpha \in \T^d$ and consider the homomorphism $\phi_\alpha : \Z^d \to S^1$ taking $v \mapsto e(v\cdot \alpha)$. Then
\begin{equation}\label{eq: index of kernel is order}
[ \Z^d : \ker \phi_\alpha] = \ord(\alpha)
\end{equation}
and
\begin{equation}\label{eq: density of kernel is 1/order}
d_{Q_N}(\ker \phi_\alpha) := \lim_{N\to \infty} \frac{\labs{\ker \phi_\alpha \cap Q_N}}{\labs{Q_N}} = \frac{1}{[\Z^d : \ker \phi_\alpha]}
\end{equation}
where $1/\infty : = 0$.
\end{lem}
\begin{proof}
We start with equation \eqref{eq: index of kernel is order}.
First consider the case that $\ord(\alpha)=\infty$. Suppose that $[\Z^d : \ker \phi_\alpha] < \infty$. By the first isomorphism theorem $\phi_\alpha(\Z^d) \cong \Z^d / \ker \phi_\alpha$ must be finite so there exists a positive integer $n$ such that $\phi_\alpha(v)^n = 1$ for all $v\in \Z^d$. But this means that $n\alpha \cdot v = 0$ in $\T^d$ for all $v \in \Z^d$, contradicting that $\ord(\alpha) = \infty$. Now if $\ord(\alpha)=q$ then there exist integers $0\leq a_1, \ldots,a_d \leq q-1$ with $\gcd(a_1,\ldots,a_d,q)=1$ such that $\alpha = (a_1,\ldots,a_d)/q$. It follows that 
\[ \phi_\alpha(\Z^d) \subset \{ e(a/q) \, : a \in \Z / q \Z\}\] and so $|\phi_\alpha(\Z^d)|\leq q$. On the other hand, B\'{e}zout's identity ensures there exists $v_1,\ldots,v_d,t \in \Z$ such that $(v_1,\ldots,v_d)\cdot(a_1,\ldots,a_d) + tq =1$, i.e. $\sum_{i=1}^d v_i a_i \equiv 1 \pmod{q}$ and so $e(1/q) \in \phi_\alpha(\Z^d)$ which implies that $|\phi_\alpha(\Z^d)|\geq q$. By the first isomorphism theorem then
\[ [\Z^d : \ker \phi_\alpha] = \labs{\Z^d / \ker \phi_\alpha} = \labs{\phi_\alpha(\Z^d)}=q.\]

We now prove equation \eqref{eq: density of kernel is 1/order}. Clearly if $H\leq \Z^d$ has infinite index then $d_{Q_N}(H) =0$, since otherwise for $\{v_i\}_{i=1}^\infty = \Z^d / H$ we can write
\begin{align*}
d_{Q_N}(\Z^d) &= d_{Q_N}\left( \bigsqcup_{i=1}^\infty H+v_i\right)\\
&\geq d_{Q_N}\left(  \bigsqcup_{i=1}^M H+v_i\right)\\
&= M d_{Q_N}\left( H\right) \to \infty \text{ as }M\to \infty.
\end{align*}
If $H\leq \Z^d$ has $[\Z^d : H ] = q$, then similarly
\[ 1 = d_{Q_N} (\Z^d) = d_{Q_N} \left( \bigsqcup_{i=1}^q H+v_i \right) = q d_{Q_N} (H)\]
which proves the lemma.
\end{proof}
For a vector $v\in \Z^d$ we define its annihilator $L_v^\perp \subset \T^d$ to be
\[ L_v^\perp : = \left \{ \alpha \in \T^d \, : \, v\cdot \alpha = 0 \text{ in }\T^d \right \}.\]
\begin{lem}\label{Lemma: formula for f}
For $\alpha \in \T^d$ consider
\[ f(\alpha) :=\lim_{N\to \infty} \frac{1}{|Q_N|} \sum_{v \in Q_N} \mathbf{1}_{L_v^\perp}(\alpha).\]
Then for every $\alpha \in \T^d$, $f(\alpha)$ exists and equals $1/\ord(\alpha)$, where $1/\infty : = 0$.
\end{lem}
\begin{proof}
Note that for fixed $\alpha$, $\mathbf{1}_{L_v^\perp}(\alpha) = \mathbf{1}_{\ker\phi_\alpha}(v)$, so by Lemma \ref{Lemma: Kernels and densities} we have that
\[ f(\alpha) = d_{Q_N}(\ker \phi_\alpha) = \frac{1}{\ord(\alpha)}.\]
\end{proof}
We will also need the follow essential fact from \cite{BjF}.
\begin{lem}[{\cite[Lemma 3.2]{BjF}}] \label{Lemma: Michael Lemma}
Given an probability preserving action $T: \Z^d \acts (X,\mu)$, a set $A\subset X$ with $\mu(A)>0$ and a vector $v\in \Z^d$, we have that
\[ \mu \left(\bigcup_{n \in \Z} T_{nv} A\right) \geq \frac{\sigma_A(\{0\})}{\sigma_A(L_v^\perp)}.\]
\end{lem}
\begin{proof}[Proof of Theorem \ref{Th: Size of expansive vectors}]
Fix some $\delta,\varepsilon,\eta>0$. Let $T:\Z^d \acts (X,\mu)$ be a probability preserving action and suppose $A\subset X$ has $\mu(A)\geq\delta$.  Suppose that $\sigma_A(\Rat(M))<\kappa$ for $M$ and $\kappa$ to be determined later. By Lemma \ref{Lemma: formula for f} and the dominated convergence theorem we have that
\begin{align}
\lim_{N\to \infty} \frac{1}{\labs{Q_N}}\sum_{v\in Q_N} \sigma_A(L_{v}^\perp) & = \lim_{N\to \infty}\int_{\T^d} \frac{1}{\labs{Q_N}}\sum_{v\in Q_N} \mathbf{1}_{L_{v}^\perp}(\alpha)\, d \sigma_A(\alpha) \nonumber \\
&=\sigma_A(\{0\}) + \sum_{q=2}^\infty \frac{1}{q} \sigma_A(\{\alpha\in \T^d \, : \, \ord(\alpha)=q\}). \label{eq: Lv perp average} \nonumber
\end{align}
Set $Y(v): = \sigma_A(L_v^\perp) - \sigma_A(\{0\})\geq 0$. The previous equation is then equivalent to
\[ \lim_{N\to \infty} \frac{1}{\labs{Q_N}}\sum_{v\in Q_N} Y(v) = \underbrace{\sum_{q=2}^\infty \frac{1}{q} \sigma_A(\{\alpha\in \T^d \, : \, \ord(\alpha)=q\})}_{:=L(A)}.\]
For $t>0$ consider
\[E(t) = \{ v\in \Z^d \, : \, Y(v)\geq t\}.\]
By Markov's inequality we have that $\overline{d}_{Q_N}(E(t))\leq L(A)/t$ which is equivalent to
\begin{equation}\label{eq: Markov on E(t) compliment}
\underline{d}_{Q_N}(E(t)^c) \geq 1-\frac{L(A)}{t}.
\end{equation}
By definition, any $v\in E(t)^c$ has that $\sigma_A(L_v^\perp)<\sigma_A(\{0\}) + t$, and so by Lemma \ref{Lemma: Michael Lemma} must satisfy that
\[ \mu\left(\bigcup_{n\in \Z}T_{nv}A\right) > 1-\frac{t}{\sigma_A(\{0\})+t} > 1-\varepsilon,\]
where the last inequality will hold provided that we pick $t$ small enough in terms of $\sigma_A(\{0\})$ and $\varepsilon$. In particular, we can take $t = \sigma_{A}(\{0\})\varepsilon$ so that
\[  \left\{v \in \Z^d \, : \, \mu\left(\bigcup_{n\in \Z}T_{nv}A\right)>1-\varepsilon\right \} \supset E(t)^c.\]
By equation \eqref{eq: Markov on E(t) compliment} then
\begin{align}
\underline{d}_{Q_N} \left(\left\{v \in \Z^d \, : \, \mu\left(\bigcup_{n\in \Z}T_{nv}A\right)>1-\varepsilon\right \}\right) &\geq \underline{d}_{Q_N}(E(t)^c) \nonumber \\
&\geq 1-\frac{L(A)}{\sigma_A(\{0\})\varepsilon}.\label{eq: ineq with LA}
\end{align}
Now pick
\[ M = \frac{2}{\delta^2 \varepsilon \eta} \qquad \text{and} \qquad \kappa = \frac{\delta^2 \varepsilon \eta}{2}.\]
Then
\begin{align}
L(A) &= \sum_{q=2}^{\infty} \frac{1}{q}\sigma_A(\{\alpha\in \T^d \, : \, \ord(\alpha)=q\}) \nonumber \\
&\leq \sigma_A(\Rat(M)) + 1/M \nonumber \\
&<\delta^2 \varepsilon \eta \leq \sigma_{A}(\{0\}) \varepsilon \eta, \label{eq: LA bound}
\end{align}
where in the last inequality we use that $\delta^2\leq \mu(A)^2 \leq \sigma_A(\{0\})$. The conclusion then follows from equations \eqref{eq: ineq with LA} and \eqref{eq: LA bound}.
\end{proof}

\section{The measure increment argument}
We now show how Corollary \ref{Cor: Expansive vectors} follows from Theorem \ref{Th: Size of expansive vectors} via the ergodic measure increment argument introduced in \cite{FS}. We recall the relevant background.
\begin{pro}[$T^k$-ergodic components {\cite[Proposition A.2]{Bu}}]\label{Prop: Tk ergodic components}
Let $T: \Z^d \acts (X,\mu)$ act ergodically. For any positive integer $k$ there exist finitely many $k\Z^d$-invariant and ergodic probability measures $\nu_1,\ldots,\nu_n$ with disjoint supports such that
\begin{equation*}
\mu = \frac{1}{n}\sum_{i=1}^n \nu_i.
\end{equation*}
In particular $\nu_i \ll \mu$ for each $i=1,\ldots,n$. We call $\nu_1,\ldots,\nu_n$ the $T^k$-ergodic components of $\mu$.
\end{pro}
\begin{lem}[{\cite[Lemma 3.3]{FS}}]\label{Lemma: Increment}
Let $T:\Z^d \acts (X,\mu)$ be an ergodic action and let $A\subset X$ have $\mu(A)>0$. For any $M\in \Z_{>0}$ there exists some $T^{M!}$-ergodic component $\nu$ of $\mu$ such that
\[\nu(A)\geq \sqrt{\mu(A)^2 + \sigma_A(\Rat(M))}.\]
\end{lem}
\begin{proof}[Proof of Corollary \ref{Cor: Expansive vectors}]
Fix $\delta,\varepsilon,\eta>0$ and let $M= M(\delta,\varepsilon,\eta)$ and $\kappa=\kappa(\delta,\varepsilon,\eta)$ be as in Theorem \ref{Th: Size of expansive vectors}. Let $T:\Z^d \acts (X,\mu)$ act ergodically and suppose $A\subset X$ has $\mu(A)\geq \delta$. If the conclusion holds with $k=1$ and $\nu=\mu$ we are done. Otherwise by Theorem \ref{Th: Size of expansive vectors} we must have that $\sigma_A(\Rat(M))\geq \kappa$, and so by Lemma \ref{Lemma: Increment} we can find some $T^{M!}$-ergodic component $\nu_1$ of $\mu$ with
\[ \nu_1(A)\geq \sqrt{\mu(A)^2 + \kappa} \geq \mu(A) + \frac{\kappa}{3}.\]
Let $k_1 = M!$ and set $S_v:= T_{k_1 v}$ for all $v\in \Z^d$. If our conclusion holds with $k=k_1$ and $\nu=\nu_1$ then we are done. Otherwise, since $S:\Z^d \acts (X,\nu_1)$ is an ergodic system and $\nu_1(A)> \mu(A)\geq \delta$ then by Theorem \ref{Th: Size of expansive vectors} and Lemma \ref{Lemma: Increment} again we can find some $T^{k_1^2}$-ergodic component $\nu_2$ of $\mu$ with
\[ \nu_2(A)\geq \nu_1(A)+\frac{\kappa}{3}\geq \mu(A) + \frac{2\kappa}{3}.\]
Repeating in this way, we must reach our conclusion with $k = k_1^j = (M!)^j$ for some $j\leq \ceil{3/\kappa}$, since otherwise we will find a probability measure $\nu$ with $\nu(A) > 1$, which is clearly a contradiction. Hence the theorem holds with $k_0 = (M!)^{\ceil{3/\kappa}}$.
\end{proof}

Theorem \ref{Theorem: Quantitative multiple twisted recurrence} follows from the exact same measure increment argument combined with the following Theorem, which states, in complete analogy with Theorem \ref{Th: Size of expansive vectors}, that the conclusion of Theorem \ref{Theorem: Quantitative multiple twisted recurrence} can always be reached with $k=1$ provided that $\sigma_A(\Rat(M))$ is sufficiently small.

\begin{thm}\label{Theorem: Small ratM gives conclusion}
For any $D\in\Z_{>0}$ and $\delta>0$ there exist some $\kappa=\kappa(D,\delta)>0$ and $M=M(D,\delta)>0$ such that the following holds. For any probability preserving action $T:\Z^d\acts (X,\mu)$ and any $A\subset X$ with $\mu(A)\geq \delta$, if $\sigma_A(\Rat(M))<\kappa$ then for any $v_1,\ldots,v_d\in \Z^d$ with $0< \labs{\det(v_1,\ldots,v_d)}\leq D$ there exists some $\gamma \in \SL_d(\Z)$ with
\begin{equation}\label{eq: Recurrence conclusion with k=1}
\mu\left( A\cap T_{\gamma v_1}A\cap \ldots \cap T_{\gamma v_d}A\right)>0.
\end{equation}
\end{thm}

\begin{rema}\label{Remark: Finite vs entire rat spec}
Implicit in their proof of Theorem \ref{Theorem: Non-quantitative multiple twisted recurrence}, the authors of \cite{BjCF} show that under the stronger spectral assumption that
\[\sigma_A\left(\{\alpha \in \T^d \, : \, 1<\ord(\alpha) < \infty\}\right)\]
is sufficiently small in terms of $\mu(A)$, then for \textit{any} linearly independent $v_1,\ldots,v_d\in \Z^d$ there exists $\gamma \in \SL_d(\Z)$ for which \eqref{eq: Recurrence conclusion with k=1} holds. Here we see the fundamental difference between Theorem \ref{Theorem: Non-quantitative multiple twisted recurrence} and Theorem \ref{Theorem: Quantitative multiple twisted recurrence}; the conclusion of Theorem \ref{Theorem: Non-quantitative multiple twisted recurrence} depends on the entire \textit{infinite} rational spectrum, and indeed Example \ref{Example: Counter example} shows that this really is unavoidable, whereas the conclusion of Theorem \ref{Theorem: Quantitative multiple twisted recurrence} depends only on the \textit{finite} piece of the rational spectrum corresponding to $\Rat(M)$.
\end{rema}
\begin{proof}[Proof of Theorem \ref{Theorem: Quantitative multiple twisted recurrence} via Theorem \ref{Theorem: Small ratM gives conclusion} and Lemma \ref{Lemma: Increment}.]
Repeat the measure increment argument used in the proof of Corollary \ref{Cor: Expansive vectors} with Theorem \ref{Theorem: Small ratM gives conclusion} replacing the role of Theorem \ref{Th: Size of expansive vectors}. 
\end{proof}

\section{A further reduction of Theorem \ref{Theorem: Small ratM gives conclusion}}
In this section we show how the conclusion of Theorem \ref{Theorem: Small ratM gives conclusion} follows from the existence of $\varepsilon$-expansive vectors $v$ which also satisfy that $\mu(A\cap T_v A)$ is large. Modulo some additional bookkeeping of quantitative constants, the argument is essentially identical to the one presented in the proof of Theorem \ref{Theorem: Non-quantitative multiple twisted recurrence} in \cite{BjCF}.
\begin{pro}\label{Prop: Existence of good vectors}
For any $D\in\Z_{>0}$ and $\delta>0$ there exist some $\kappa=\kappa(D,\delta)>0$ and $M=M(D,\delta)>0$ such that for any probability preserving action $T:\Z^d\acts (X,\mu)$ and any $A\subset X$ with $\mu(A)\geq \delta$ and $\sigma_A(\Rat(M))<\kappa$ the following holds. For all $m=1,\ldots, D$ there exists some $v\in m \mathcal{P}$ with
\begin{equation}\label{eq: big return time}
\mu(A\cap T_{v}A) > \frac{\mu(A)^2}{2}
\end{equation}
and
\begin{equation}\label{eq: first D multiples are expansive}
\mu\left(\bigcup_{n\in \Z} T_{nlv}A\right) > 1-\frac{\delta^2}{2(d-1)} \quad \text{for all } l=1,\ldots,D.
\end{equation}
\end{pro}
\begin{proof}[Proof of Theorem \ref{Theorem: Small ratM gives conclusion} via Proposition \ref{Prop: Existence of good vectors}]
Fix some positive integer $D$ and $\delta>0$. Let $\kappa=\kappa(D,\delta)$ and $M=M(D,\delta)$ be as in the statement of Proposition \ref{Prop: Existence of good vectors}. We will show that this $\kappa$ and $M$ satisfy the conclusion of Theorem \ref{Theorem: Small ratM gives conclusion}. So let $T:\Z^d \acts (X,\mu)$ be a probability preserving system and suppose $A\subset X$ has $\mu(A)\geq \delta$ and $\sigma_A(\Rat(M))<\kappa$. Let $v_1,\ldots,v_d \in \Z^d$ have $0 < \labs{\det(v_1,\ldots,v_d)}\leq D$. Denote $L:=\labs{\det(v_1,\ldots,v_d)}$. Since $\gcd(v_1) \leq L \leq D$, then by Proposition \ref{Prop: Existence of good vectors} there exists $v\in \gcd(v_1) \mathcal{P}$ satisfying equations \eqref{eq: big return time} and \eqref{eq: first D multiples are expansive}. It is easy to see that $\SL_d(\Z)$ acts transitively on $\mathcal{P}$ so there exists some $\gamma_0 \in \SL_d(\Z)$ with $\gamma_0 v_1 = v$. Since $\mu(A)\geq \delta$ then equation \eqref{eq: big return time} then implies that
\[ \mu(A\cap T_{\gamma_0 v_1}A)> \frac{\delta^2}{2},\] 
and since $T$ preserves $\mu$, the $l=L$ case of equation \eqref{eq: first D multiples are expansive} implies that
\[ B_j:= \bigcup_{n\in \Z} T_{nL \gamma_0 v_1 + \gamma_0 v_j}A\qquad \text{has} \qquad \mu(B_j) > 1-\frac{\delta^2}{2(d-1)}\]
for all $j=2,\ldots, d$.
Then
\begin{align*}
\mu\left( (A\cap T_{\gamma_0 v_1}A) \cap \bigcap_{j=2}^d B_j \right)&=1 - \mu\left((A\cap T_{\gamma_0 v_1}A)^c \cup \bigcup_{j=2}^d B_j^c \right)\\
&>\mu(A\cap T_{\gamma_0 v_1}A) - \sum_{j=2}^d \mu(B_j^c)\\
&> \frac{\delta^2}{2} - (d-1) \frac{\delta^2}{2(d-1)} = 0.
\end{align*}
It then follows from the definition of each $B_j$ that there exist $n_2,\ldots,n_d \in \Z$ such that
\[ \mu\left(A\cap T_{\gamma_0 v_1}A\cap T_{n_2L \gamma_0 v_1 + \gamma_0v_2}A\cap \ldots \cap T_{n_d L \gamma_0 v_1 + \gamma_0 v_d} A\right)>0.\]
All that remains is to check that there exists some $\gamma \in \SL_d(\Z)$ with
\[\gamma v_1 = \gamma_0 v_1 \qquad\text{and} \qquad \gamma v_j = \gamma_0 v_j + n_j L \gamma_0 v_1 \quad \text{for all } j =2 ,\ldots, d.\] 
Set $u_i:= \gamma_0 v_i$ for all $i=1,\ldots,d$ and for each $j=2,\ldots, d$ define a $\Q$-linear map on $\Q^d$ by setting
\[ S_j u_i = u_i + \delta_{ij}L u_1 \qquad \text{for all } i=1,\ldots,d \] 
where $\delta_{ij}$ is the Kronecker delta, and extending by linearity. We claim that that each $S_j \in \SL_d(\Z)$ and so
\[ \gamma: = S_2^{n_2}S_3^{n_3}\ldots S_d^{n_d} \gamma_0 \in \SL_d(\Z)\]
is as required.
Each $S_j$ is a shear and so clearly has $\det(S_j)=1$, so it suffices to prove that $S_j$ has integer entries. Let $U\in M_{d\times d}(\Z)$ have columns $u_1,\ldots,u_d$ and notice that $\det(U) = L$. Each $S_j$ is then of the form 
\[S_j = U (I + L E_{1j}) U^{-1}\] where $E_{1j}$ is the elementary matrix with a $1$ in position $(1,j)$ and zeros elsewhere. Now $U^{-1} = \frac{1}{L} \mathrm{adj}(U)$ where $\mathrm{adj}(U)\in M_{d\times d}(\Z)$ so
\[ S_j = U(I+L E_{1j})U^{-1} = I + UL E_{1j}U^{-1}= I +U E_{1j} \mathrm{adj}(U)\]
which implies that $S_j$ has integer entries as required.
\end{proof}

\section{A proof of Proposition \ref{Prop: Existence of good vectors}}
We prove Proposition \ref{Prop: Existence of good vectors} by first estimating the size of the set of vectors satisfying equations \eqref{eq: big return time} and \eqref{eq: first D multiples are expansive} separately. Non-emptiness of the intersection of these two sets will then follow from size considerations alone.
\begin{pro}\label{Prop: Density of big return times}
For any $D\in \Z_{>0}$ and $\delta>0$ there exist some $\kappa=\kappa(D,\delta)>0$ and $M=M(D,\delta)>0$ such that for any probability preserving action $T:\Z^d\acts (X,\mu)$ and any $A\subset X$ with $\mu(A)\geq \delta$ and $\sigma_A(\Rat(M))<\kappa$ the following is true. For all $m=1,\ldots,D$ the set
\[
R(m):= \left\{v \in \Z^d \, : \, \mu(A\cap T_{mv} A) > \frac{\mu(A)^2}{2}\right\}
\]
satisfies that
\[\underline{d}_{Q_N^\mathcal{P}}(R(m))>\frac{\mu(A)^2}{3}.\] 
\end{pro}
\begin{proof}
Let $\delta>0$ and $D\in \Z_{>0}$. Let $M$ and $\kappa$ be positive constants to be later determined. Let $T: \Z^d \acts (X,\mu)$ be probability preserving and suppose $A\subset X$ satisfies that $\mu(A)\geq \delta$ and $\sigma_A(\Rat(M))<\kappa$. Fix some $m\in \{1,\ldots,D\}$ and denote by $\sigma_A^m$ the pushforward of $\sigma_A$ under the map $\alpha \mapsto m \alpha$. Define the average
\[ S_{N}(A):= \frac{1}{\labs{Q_N^\mathcal{P}}} \sum_{v\in Q_N^\mathcal{P}} \mu(A\cap T_{mv}A).\]
Using the definition of $\sigma_A$ we have that
\begin{equation*}
S_{N}(A) = \frac{1}{\labs{Q_N^\mathcal{P}}} \sum_{v\in Q_N^\mathcal{P}} \int_{\T^d} e(mv \cdot \alpha) \, d \sigma_A(\alpha) = \int_{\T^d}  \frac{1}{\labs{Q_N^\mathcal{P}}} \sum_{v\in Q_N^\mathcal{P}} e(v\cdot \alpha) \, d \sigma_A^m(\alpha).
\end{equation*}
By Proposition \ref{Prop: Formula for exp sums over primitives} and the dominated convergence theorem it then follows that
\begin{equation}
\lim_{N\to \infty} S_{N}(A) = \sigma_A^m(\{0\}) + \sum_{q \geq 2} \frac{\boldsymbol{\mu}(q)}{J_d(q)} \sigma_A^m(\{ \alpha \in \T^d \, : \, \ord(\alpha)=q\}).
\end{equation}
Let $\varepsilon>0$ be later determined. Clearly\footnote{In fact it follows immediately from the Euler product formulae in equations \eqref{eq: Euler prod for zeta} and \eqref{eq: euler prod for Jordan totient} in Section \ref{Section: Exp sums} that $J_d(q)\zeta(d) \geq q^d$.} we have that $J_d(q) \to \infty $ as $ q\to \infty$, and since $\sigma_A^m$ is a finite measure then there exists $Q=Q(\varepsilon)>0$ such that
\begin{equation}\label{eq: Tail bound}
\labs{\sum_{q >Q} \frac{\boldsymbol{\mu}(q)}{J_d(q)} \sigma_A^m(\{ \alpha \in \T^d \, : \, \ord(\alpha)=q\})}<\frac{\varepsilon}{2}.
\end{equation}
For any $\alpha \in \T^d$ with finite order, it is not hard to show that $\ord(m\alpha) = \ord(\alpha)/\gcd(\ord(\alpha),m)$, and since $\gcd(\ord(\alpha),m)\leq m \leq D$ it follows that $\sigma_A^m(\Rat(Q))\leq \sigma_A(\Rat(DQ))$. We can then calculate
\begin{align}
\labs{\sum_{q=2}^Q \frac{\boldsymbol{\mu}(q)}{J_d(q)} \sigma_A^m(\{ \alpha \in \T^d \, : \, \ord(\alpha)=q\})}&\leq \sigma^m_A(\Rat(Q)) \nonumber\\
&\leq \sigma_A(\Rat(DQ)) < \frac{\varepsilon}{2}\label{eq: Middle bound}
\end{align}
where the last inequality follows provided we chose $M = DQ$ and $\kappa = \varepsilon/2$. Now we also have that
\begin{equation}\label{eq: Atom bound}
\sigma_A^m(\{0\})\geq \sigma_A(\{0\}) = \mu(A)^2
\end{equation}
so together with the reverse triangle inequality, equations \eqref{eq: Tail bound}, \eqref{eq: Middle bound}, and \eqref{eq: Atom bound} imply that
\begin{equation}\label{eq: return time average lower bound}
\lim_{N\to \infty} S_{N}(A) \geq \mu(A)^2 - \varepsilon.
\end{equation}
Write
\[ \delta_N:= \frac{\labs{R(m)\cap Q_N^\mathcal{P}} }{\labs{Q_N^\mathcal{P}}}\]
and estimate
\[ S_N(A) \leq \delta_N \mu(A) + \frac{\mu(A)^2}{2}(1-\delta_N) = \delta_N\left(\mu(A) - \frac{\mu(A)^2}{2}\right)  + \frac{\mu(A)^2}{2}.  \]
By taking the limit infimum of both sides and applying equation \eqref{eq: return time average lower bound} we then have that
\begin{equation}
\frac{\mu(A)^2}{2} - \varepsilon \leq \underline{d}_{Q_N^\mathcal{P}}(R(m)) \left(\mu(A) - \frac{\mu(A)^2}{2}\right)< \underline{d}_{Q_N^\mathcal{P}}(R(m)).
\end{equation}
Taking $\varepsilon = \delta^2/6$ then ensures that $\underline{d}_{Q_N^\mathcal{P}}(R(m))>\mu(A)^2/3$, and since $m=1,\ldots,D$ was arbitrary we are done.
\end{proof}

\begin{pro}\label{Prop: Density of expansive vectors}
For any $D\in \Z_{>0}$ and any $\delta, \varepsilon,\eta>0$ there exist some $M=M(D,\delta,\varepsilon,\eta)>0$ and $\kappa=\kappa(D,\delta,\varepsilon,\eta)>0$ such that for any probability preserving action $T:\Z^d\acts (X,\mu)$ and any $A\subset X$ with $\mu(A)\geq \delta$ the following is true. If $\sigma_A(\Rat(M))<\kappa$ then
\[
\underline{d}_{Q_N}\left( \left\{v \in \Z^d \, : \, \mu\left( \bigcup_{n\in \Z}T_{nlv}A\right) >1- \varepsilon\quad  \text{ for each }l=1,\ldots,D\right\} \right)>1-\eta.
\]
\end{pro}
\begin{proof}
Fix $D\in \Z_{>0}$ and $\delta, \varepsilon,\eta>0$. Let $M_1 = M_1(\delta,\varepsilon,\eta/D)$ and $\kappa_1 = \kappa_1(\delta,\varepsilon,\eta/D)$ be as in Theorem \ref{Th: Size of expansive vectors}. We claim that
\[ M: = DM_1 \qquad \text{and }\qquad \kappa:= \kappa_1\]
are the desired constants.

So let $T:\Z^d\acts (X,\mu)$ be probability preserving and suppose $A\subset X$ satisfies that $\mu(A)\geq \delta$ and $\sigma_A(\Rat(M))<\kappa$. Fix $l\in \{1,\ldots,D\}$ and consider the sub-action $S$ defined by $S_v:= T_{lv}$ for each $v\in \Z^d$. Then the spectral measure of $A$ with respect to $S:\Z^d \acts (X,\mu)$ is precisely $\sigma^l_A$, the pushforward of $\sigma_A$ (where $\sigma_A$ is the spectral measure of $A$ with respect to $T:\Z^d \acts (X ,\mu)$) under the map $\alpha \mapsto l\alpha$.
Just as in the proof of Proposition \ref{Prop: Density of big return times} we have that
\[ \sigma^l_A(\Rat(M_1)) \leq \sigma_A(\Rat(DM_1)) = \sigma_A(\Rat(M)<\kappa\]
and so by Theorem \ref{Th: Size of expansive vectors} applied to the action $S:\Z^d \acts (X,\mu)$ we have that
\[\underline{d}_{Q_N}\left(\underbrace{ \left\{v \in \Z^d \, : \, \mu\left( \bigcup_{n\in \Z}T_{nlv}A\right) >1- \varepsilon\right\}}_{:=E(l,\varepsilon)} \right)>1-\frac{\eta}{D}.\]
As $l\in \{1,\ldots,D\}$ was arbitrary then 
\[ \underline{d}_{Q_N}\left(\bigcap_{l=1}^D E(l,\varepsilon)\right) > 1-\eta\]
as required.
\end{proof}

\begin{proof}[Proof of Proposition \ref{Prop: Existence of good vectors} using Propositions \ref{Prop: Density of big return times} and \ref{Prop: Density of expansive vectors}] 
Let $\delta>0$ and $D\in \Z_{>0}$. Let $M_1 = M_1(D,\delta)$ and $\kappa_1=\kappa_1(D,\delta)$ be as in Proposition \ref{Prop: Density of big return times}, and let $M_2=M_2(D,\delta,\varepsilon,\eta)$ and $\kappa_2=\kappa_2(D,\delta,\varepsilon,\eta)$ be as in Proposition \ref{Prop: Density of expansive vectors} with the choices $\varepsilon=\delta^2/(2(d-1))$ and $\eta = \delta^2/(3\zeta(d))$.

We claim that
\[M := \max(M_1,M_2)\qquad \text{and} \qquad \kappa := \min(\kappa_1,\kappa_2)\]
will be the constants required in the statement of Proposition \ref{Prop: Existence of good vectors}.

So let $T: \Z^d \acts (X,\mu)$ be probability preserving and suppose $A\subset X$ has $\mu(A)\geq \delta$ and satisfies that $\sigma_A(\Rat(M))<\kappa$. Then by Proposition \ref{Prop: Density of big return times}, we have that
\begin{equation}\label{eq: Size of R(M)}
\underline{d}_{Q_N^\mathcal{P}}\left(\underbrace{\left\{ v \in \Z^d \, : \, \mu(A\cap T_{mv}A)>\frac{\mu(A)^2}{2}\right\}}_{=R(m)}\right) > \frac{\delta^2}{3}
\end{equation}
for each $m=1,\ldots,D$. By Proposition \ref{Prop: Density of expansive vectors} we have that
\[ E:= \left\{v \in \Z^d \, : \, \mu\left( \bigcup_{n\in \Z}T_{nlv}A\right) >1- \frac{\delta^2}{2(d-1)}\quad   \text{ for each }l=1,\ldots,D\right\}\]
satisfies that
\begin{equation}\label{eq: Size of E}
\underline{d}_{Q_N}(E) >1- \frac{\delta^2}{3\zeta(d)}. 
\end{equation}

By inclusion-exclusion we have that
\[ \labs{E\cap  \mathcal{P}\cap Q_N} \geq \labs{E\cap Q_N} +\labs{\mathcal{P}\cap Q_N} -\labs{Q_N}. \]
Dividing by $\labs{Q_N}$ and taking the limit infimum of both sides then shows that
\begin{equation}\label{eq: Using inclusion exclusion for density}
 \underline{d}_{Q_N}(E\cap \mathcal{P}) \geq \underline{d}_{Q_N}(E) + \underline{d}_{Q_N}(\mathcal{P}) -1> \frac{1}{\zeta(d)}\left(1-\frac{\delta^2}{3}\right)
\end{equation}
where in the last inequality we use equations \eqref{eq: Primitive density exists} and \eqref{eq: Size of E}.
On the other hand, equation \eqref{eq: Primitive density exists} also implies that
\[ \lim_{N\to \infty} \frac{\labs{\mathcal{P}\cap Q_N}}{\labs{Q_N}} = \frac{1}{\zeta(d)},\]
which together with equation \eqref{eq: Using inclusion exclusion for density} implies that
\begin{equation}\label{eq: Relative size of E}
\underline{d}_{Q_N^\mathcal{P}}(E) > 1-\frac{\delta^2}{3}.
\end{equation}
Together then equations \eqref{eq: Size of R(M)} and \eqref{eq: Relative size of E} ensure that
\[\underline{d}_{Q_N^\mathcal{P}}(E) + \underline{d}_{Q_N^\mathcal{P}}(R(m))>1 \qquad \text{ for all }m=1,\ldots,D, \]
and so $E\cap R(m)\cap \mathcal{P}\neq \emptyset$ for each $m=1\ldots,D$ as required.
\end{proof}

\section{Example \ref{Example: Counter example}}
We show that for every positive integer $k_0$ there exists an ergodic system $T:\Z^d \acts (X,\mu)$ and a set $A\subset X$ with $\mu(A)=2^{-d}$ such that for each $k=1,\ldots,k_0$, there exists some $v_k \in \Z^d$ for which
\[ \mu(A\cap T_{\gamma kv_k}A)=0 \quad \text{for every }\gamma \in \SL_d(\Z).\]

So fix $k_0$ and set $L:= k_0!$. Let $X=(\Z/(2L\Z))^d$ with counting probability measure $\mu$, and let $T$ be the $\Z^d$ action on $X$ by translations. Clearly $T$ acts ergodically. Consider the set $A:=\{0,1,\ldots,L-1\}^d\subset X$. Clearly $A$ has $\mu(A)=2^{-d}$. For any $v\in \mathcal{P}$ we have that $v \not \equiv (0,\ldots,0) \pmod{2}$ and so at least one component of $Lv$ is congruent to $L\pmod{2L}$. It follows that
\[A\cap T_{Lv}A = \emptyset \qquad \text{for all }v\in \mathcal{P}.\]
Since the $SL_{d}(\Z)$ action on $\Z^d$ preserves the $\gcd$ of any vector, for each $k=1,\ldots,k_0$, and any $w\in \mathcal{P}$ the vector $v = (L/k)w$ then has
\[ \mu(A\cap T_{\gamma kv}A) = 0\quad \text{for every }\gamma \in \SL_d(\Z).\]
\section{Exponential sums for primitive vectors}\label{Section: Exp sums}
In this section we prove Proposition \ref{Prop: Formula for exp sums over primitives}, restated here for convenience.
\begin{pro}
For any $\alpha \in \T^d$ we have that
\[g(\alpha): = \lim_{N\to \infty}\frac{1}{\labs{Q_N^\mathcal{P}}} \sum_{v\in Q_N^{\mathcal{P}}} e(v\cdot \alpha)=
\begin{cases}
1 & \text{ if } \alpha=0\\
\frac{1}{J_d(q)}\boldsymbol{\mu}(q) & \text{ if } \ord(\alpha)=q\geq 2\\
0 & \text{ if }\ord(\alpha)=\infty.
\end{cases}\]
\end{pro}
We first recall some facts from analytic number theory, all of which can be found in \cite{A}.

The M\"{o}bius function $\boldsymbol{\mu}: \Z_{>0}\to \R$ is defined by
\[
\boldsymbol{\mu}(n) : = 
\begin{cases}
1 & \text{if }n=1\\
(-1)^k &\text{if }n \text{ is a product of }k\text{ distinct primes}\\
 0 &\text{otherwise}
\end{cases}
\]
and satisfies that for any positive integer $m$,
\begin{equation}\label{eq: Mobius inversion}
\sum_{n | m} \boldsymbol{\mu}(n) = 
\begin{cases}
1 & \text{if }m=1\\
 0 &\text{otherwise.}
\end{cases}
\end{equation}
The Riemann zeta function is defined by
\[ \zeta(s) := \sum_{n=1}^\infty \frac{1}{n^s} \qquad \text{for any }s\in \C \text{ with }\mathrm{Re}(s)>1\]
and satisfies the Euler product formula
\begin{equation}\label{eq: Euler prod for zeta}
\zeta(s) = \prod_{p\text{ prime}}\frac{1}{1-p^{-s}}\qquad \text{for any }s\in \C \text{ with }\mathrm{Re}(s)>1,
\end{equation}
and that
\begin{equation}\label{eq: Dirichlet series formula}
\frac{1}{\zeta(s)} = \sum_{n=1}^\infty\frac{\boldsymbol{\mu}(n)}{n^s}\qquad \text{for any }s\in \C \text{ with }\mathrm{Re}(s)>1.
\end{equation}

The Jordan totient function $J_q: \Z_{>0}\to \R$ is defined by 
\[ J_d(q): = \labs{\left\{ a=(a_1,\ldots,a_d)\in [q]^d\, : \, \gcd(a,q) = 1\right\}}\]
where $[q]:=\{1,\ldots,q\}$.
\begin{lem}\label{Lemma: Some NT formulae}
Let $N$ and $M$ be positive integers and let $r \in (\Z/M\Z)^d$ have $\gcd(r,M)=1$. Then we have that
\begin{equation}\label{eq: Density of primitives}
\labs{Q_N^\mathcal{P}} = \frac{\labs{Q_N}}{\zeta(d)} + O\left(\log(N)N^{d-1}\right)
\end{equation}
and
\begin{equation}\label{eq: Counting primtives in congruence class}
\labs{\left\{v\in Q_N^\mathcal{P} \, : \, v \equiv r \:(\mathrm{mod }\,M) \right\}} = \frac{\labs{Q_N}}{\zeta(d)J_d(M)} + O\left( \frac{\log(N)N^{d-1}}{M^{d-1}}\right).
\end{equation}
\end{lem}
\begin{proof}
Let $M$ be a positive integer. For any $r_0 \in \Z/M\Z$ clearly we have that
\begin{equation*}
\labs{ \left\{ n\in [-N,N]\cap \Z \, : \, n \equiv r_0 \:(\text{mod }M) \right\}} = \frac{\labs{2N+1}}{M} + O\left( 1 \right),
\end{equation*}
and so for any $r\in (\Z/M\Z)^d$
\begin{equation}\label{eq: (1)}
\labs{\left\{ v\in Q_N \, : \, v \equiv r \:(\text{mod }M) \right\}} = \frac{\labs{Q_N}}{M^d} + O\left(\frac{N^{d-1}}{M^{d-1}}\right).
\end{equation}
Taking $r\equiv 0$ in \eqref{eq: (1)} yields
\begin{equation}\label{eq: (2)}
\labs{\left\{ v\in Q_N \, : \, M \, |\, v \right\}} = \frac{\labs{Q_N}}{M^d} + O\left(\frac{N^{d-1}}{M^{d-1}}\right).
\end{equation}
Use equation \eqref{eq: Mobius inversion} to write for any $v\in \Z^d$,
\begin{equation} \label{eq: (4)}
\mathbf{1}_{\gcd(v)=1} = \sum_{n | \gcd(v)} \boldsymbol{\mu}(n) = \sum_{n\geq 1} \boldsymbol{\mu}(n) \mathbf{1}_{n | v}.
\end{equation}
We can then use \eqref{eq: (2)} and \eqref{eq: (4)} to calculate
\begin{align*}
\labs{Q_N^\mathcal{P}}&= \sum_{v \in Q_N} \mathbf{1}_{\gcd(v)=1} = \sum_{v \in Q_N}\sum_{n\geq 1} \boldsymbol{\mu}(n) \mathbf{1}_{n | v}\\
&=\sum_{v\in Q_N} \sum_{n=1}^N \boldsymbol{\mu}(n) \mathbf{1}_{n | v}\\
&= \sum_{n=1}^N \boldsymbol{\mu}(n) \labs{\left\{ v\in Q_N \, : \, n \,|\, v \right\}}\\
&=\sum_{n=1}^N \boldsymbol{\mu}(n)\left(\frac{\labs{Q_N}}{n^d} + O\left(\frac{N^{d-1}}{n^{d-1}}\right).\right)\\
&= \labs{Q_N} \sum_{n=1}^N \frac{\boldsymbol{\mu}(n)}{n^d} + \sum_{n=1}^N \boldsymbol{\mu}(n) O\left(\frac{N^{d-1}}{n^{d-1}}\right)\\
&=\labs{Q_N}\left( \sum_{n=1}^\infty\frac{\boldsymbol{\mu}(n)}{n^d} + O\left(\frac{1}{N^{d-1}}\right)\right) + O(\log(N)N^{d-1})
\end{align*}
where in the final line we bound the absolute value of the tail $\sum_{n>N}\frac{\boldsymbol{\mu}(n)}{n^d}$ by an integral and use the crude bound
\[ \labs{\sum_{n=1}^N \frac{\boldsymbol{\mu}(n)}{n^{d-1}}} \leq \sum_{n=1}^N \frac{1}{n} = O(\log(N))\]
for any $d\geq 2$. Applying the formula from \eqref{eq: Dirichlet series formula} we have that
\[ \labs{Q_N^\mathcal{P}}=\labs{Q_N}\left( \frac{1}{\zeta(d)} + O\left(\frac{1}{N^{d-1}}\right)\right) + O(\log(N)N^{d-1})\]
and collecting error terms yields equation \eqref{eq: Density of primitives}.

We now turn to proving equation \eqref{eq: Counting primtives in congruence class}. So let $r\in (\Z/M\Z)^d$ have $\gcd(r,M)=1$ and use equation \eqref{eq: Mobius inversion} again to write
\begin{align*}
\labs{\left\{v\in Q_N^\mathcal{P} \, : \, v \equiv r \:(\text{mod }M) \right\}} &= \sum_{v\in Q_N} \mathbf{1}_{v \equiv r (\text{mod }M)} \mathbf{1}_{\gcd(v)=1}  \\
&=\sum_{v\in Q_N} \mathbf{1}_{v \equiv r (\text{mod }M)} \sum_{n=1}^N \boldsymbol{\mu}(n) \mathbf{1}_{n | v}.
\end{align*}
Notice that since $\gcd(r,M)=1$ then $n \, | \, v$ implies that $\gcd(n,M)=1$. Indeed writing $v = r + Ml$ for some $l\in \Z^d$ then
\[ n \, | \, v \implies \gcd(n,M) \, | \, v \implies \gcd(n,M) \, | \, r\implies \gcd(n,M)=1.\]
Continuing with our earlier calculation then
\begin{align}
\labs{\left\{v\in Q_N^\mathcal{P} \, : \, v \equiv r \:(\text{mod }M) \right\}} &= \sum_{v\in Q_N} \mathbf{1}_{v \equiv r (\text{mod }M)} \sum_{n=1}^N \boldsymbol{\mu}(n) \mathbf{1}_{\gcd(n,M)=1}\mathbf{1}_{n | v} \nonumber \\
&= \sum_{n=1}^N \mathbf{1}_{\gcd(n,M)=1} \boldsymbol{\mu}(n) \sum_{v\in Q_N} \mathbf{1}_{v \equiv 0 (\text{mod }n)} \mathbf{1}_{v\equiv r (\text{mod }M)}.\label{eq: (5)}
\end{align}
Now for each $n$ and $M$ with $\gcd(n,M)=1$ the Chinese remainder theorem implies there exists some $r_n \in (\Z/ (nM)\Z)^d$ so that
\begin{equation}\label{eq: (6)}
\sum_{v\in Q_N} \mathbf{1}_{v \equiv 0 (\text{mod }n)} \mathbf{1}_{v\equiv r (\text{mod }M)}=\Big|\left\{v \in Q_N \, : \, v\equiv r_n \pmod{nM}\right\}\Big|.
\end{equation}
Equation \eqref{eq: (1)} provides us with a bound for the set in the right hand side of \eqref{eq: (6)}, which combined with equation \eqref{eq: (5)} allows us to calculate
\begin{align}
\Big|\{v\in Q_N^\mathcal{P} &\, : \, v \equiv  r \:(\text{mod }M) \}\Big| \nonumber\\
&=\sum_{n=1}^N \mathbf{1}_{\gcd(n,M)=1} \boldsymbol{\mu}(n) \left(\frac{\labs{Q_N}}{n^d M^d} + O\left(\frac{N^{d-1}}{n^{d-1}M^{d-1}}\right) \right)\nonumber\\
&= \frac{\labs{Q_N}}{M^d}\sum_{n=1}^N \mathbf{1}_{\gcd(n,M)=1} \frac{\boldsymbol{\mu}(n)}{n^d}+O\left(\frac{\log(N)N^{d-1}}{M^{d-1}}\right)\nonumber\\
&=\frac{\labs{Q_N}}{M^d}\left(\sum_{n=1}^\infty \mathbf{1}_{\gcd(n,M)=1} \frac{\boldsymbol{\mu}(n)}{n^d} + O\left(\frac{1}{N^{d-1}}\right)\right)+O\left(\frac{\log(N)N^{d-1}}{M^{d-1}}\right)\nonumber\\
&= \frac{\labs{Q_N}}{M^d}\sum_{n=1}^\infty \mathbf{1}_{\gcd(n,M)=1} \frac{\boldsymbol{\mu}(n)}{n^d} + O\left(\frac{\log(N)N^{d-1}}{M^{d-1}}\right). \label{eq: (7)}
\end{align}
We claim that
\begin{equation}
\frac{1}{M^d}\sum_{n=1}^\infty \mathbf{1}_{\gcd(n,M)=1} \frac{\boldsymbol{\mu}(n)}{n^d} = \frac{1}{\zeta(d) J_d(M)},
\end{equation}
which, once combined with equation \eqref{eq: (7)}, will yield equation \eqref{eq: Counting primtives in congruence class} as required. Indeed, by equation \eqref{eq: Mobius inversion} again we have that
\begin{align*}
J_d(M)= \sum_{a \in [M]^d} \mathbf{1}_{\gcd(a,M)=1}
&= \sum_{a\in [M]^d} \sum_{n \,| \,M} \boldsymbol{\mu}(n) \mathbf{1}_{n \, | \, \gcd(a,M)}\\
&=\sum_{n \,| \,M} \boldsymbol{\mu}(n)\labs{\left\{ a\in [M]^d \, : \, n \, | \, a \right\}  }\\
&=\sum_{n \,| \,M} \boldsymbol{\mu}(n) \frac{M^d}{n^d}\\
&=M^d \sum_{n\geq 1} \mathbf{1}_{n\, | \, M} \frac{\boldsymbol{\mu}(n)}{n^d}.
\end{align*}
The function $f(n)= \mathbf{1}_{n\, | \, M} \boldsymbol{\mu}(n)/n^d$ is multiplicative with an absolutely convergent sum $\sum_{n\geq 1}f(n)$ so we can use an Euler product\footnote{Recall that a function $f: \Z_{>0} \to \C$ is multiplicative if $f(nm)= f(n)f(m)$ for all $n$ and $m$ with $\gcd(n,m)=1$. If a multiplicative function $f$ has that the series $\sum_{n \geq 1}f(n)$ is absolutely convergent, then the series is equal to its Euler product \[
\sum_{n \geq 1}f(n) = \prod_{p\text{ prime}}(1+f(p)+f(p^2)+\ldots).\] See \cite[Theorem 11.6]{A} for more details.} to write
\begin{align*}
\sum_{n\geq 1} \mathbf{1}_{n\, | \, M} \frac{\boldsymbol{\mu}(n)}{n^d}&= \prod_{p}\left( 1+ \sum_{n\geq 1}\frac{\mathbf{1}_{p^n \, | \, M} \boldsymbol{\mu}(p^n)}{p^{nd}}\right)\\
&=\prod_{p} \left(1-\frac{\mathbf{1}_{p \, | \, M}}{p^d}\right)\\
&=\prod_{p \, | \, M} \left(1-\frac{1}{p^d}\right),
\end{align*}
and so
\begin{equation}\label{eq: euler prod for Jordan totient}
J_d(M) = M^d\prod_{p \, | \, M} \left(1-\frac{1}{p^d}\right).
\end{equation}
Similarly we have that
\begin{align*}
\frac{1}{M^d}\sum_{n\geq 1} \mathbf{1}_{\gcd(n,M)=1} \frac{\boldsymbol{\mu}(n)}{n^d} &= \frac{1}{M^d} \prod_{p}\left( 1+ \sum_{n\geq 1}\frac{\mathbf{1}_{\gcd(p^n,M)=1} \boldsymbol{\mu}(p^n)}{p^{nd}}\right)\\
&=\frac{1}{M^d} \prod_{p}\left( 1-\frac{\mathbf{1}_{p \,\nmid \,M}}{p^d}\right)\\
&=\frac{1}{M^d} \prod_{p \, \nmid \, M}\left(1-p^{-d}\right).
\end{align*}
Combining equations \eqref{eq: Euler prod for zeta} and \eqref{eq: euler prod for Jordan totient} we then have that
\[
\frac{1}{M^d} \prod_{p \, \nmid \, M}\left(1-p^{-d}\right)
=\frac{1}{M^d} \frac{\prod_{p}\left(1-p^{-d}\right)}{\prod_{p \, | \, M}\left(1-p^{-d}\right)}\\
=\frac{1}{\zeta(d)J_d(M)}\]
as required.
\end{proof}
\begin{proof}[Proof of Proposition \ref{Prop: Formula for exp sums over primitives}]
For $\alpha \in \T^d$ define
\begin{equation}\label{eq: Def of gN}
g_N(\alpha):=\frac{1}{\labs{Q_N^\mathcal{P}}} \sum_{v\in Q_N^{\mathcal{P}}} e(v\cdot \alpha).
\end{equation}
First consider the case that $\ord(\alpha)=q$ for some positive integer $q$. Clearly if $\alpha = 0$ the conclusion is obvious, so take $q\geq 2$. We can then write $\alpha = a/q$ for some $a\in [q]^d$ with $\gcd(a,q)=1$. Then $v \mapsto e((v\cdot \alpha)/q)$ is constant on residue classes modulo $q\Z^d$, so we can use \eqref{eq: Counting primtives in congruence class} from Lemma \ref{Lemma: Some NT formulae} to write
\begin{align}
g_N(\alpha) &= \frac{1}{\labs{Q_N^\mathcal{P}}}\sum_{v\in Q_N^\mathcal{P}} e \left( \frac{v\cdot a}{q}\right) \nonumber \\
&= \frac{1}{\labs{Q_N^\mathcal{P}}} \sum_{r \in [q]^d \: \gcd(r,q)=1}\labs{\left\{v\in Q_N^\mathcal{P} \, : \, v \equiv r \:(\mathrm{mod }\,q) \right\}}e \left( \frac{r\cdot a}{q}\right) \nonumber \\
&= \frac{1}{\labs{Q_N^\mathcal{P}}} \left(\frac{\labs{Q_N}}{\zeta(d)J_d(q)} + O\left( \frac{\log(N)N^{d-1}}{q^{d-1}}\right) \right)\sum_{\substack{r \in [q]^d \\ \gcd(r,q)=1}} e \left( \frac{r\cdot a}{q}\right). \label{eq: Reducing to Ramanujan sum}
\end{align}
By equation \eqref{eq: Density of primitives} we have that
\begin{equation}\label{eq: little o equation}
\frac{1}{\labs{Q_N^\mathcal{P}}} \left(\frac{\labs{Q_N}}{\zeta(d)J_d(q)} + O\left( \frac{\log(N)N^{d-1}}{q^{d-1}}\right) \right)
=\frac{1}{J_d(q)}\frac{1}{1+o_{N}(1)} + o_{N}(1).
\end{equation}
The sum in equation \eqref{eq: Reducing to Ramanujan sum} is a classical $d$-dimensional Ramanujan sum which we can calculate using equation \eqref{eq: Mobius inversion} again as follows.
\begin{align}
\sum_{\substack{r \in [q]^d \\ \gcd(r,q)=1}} e \left( \frac{r\cdot a}{q}\right)&=\sum_{r \in [q]^d } \mathbf{1}_{\gcd(r,q)=1} e \left( \frac{r\cdot a}{q}\right)  \nonumber \\
&=\sum_{r \in [q]^d } \sum_{t \, | \, r,q} \boldsymbol{\mu}(t) e \left( \frac{r\cdot a}{q}\right) \nonumber\\
&=\sum_{r \in [q]^d } \sum_{t \, | \, q} \mathbf{1}_{t \, | \,r} \boldsymbol{\mu}(t) e \left( \frac{r\cdot a}{q}\right) \nonumber\\
&=\sum_{t \, | \, q} \boldsymbol{\mu}(t) \sum_{\substack{r \in [q]^d \\ t \, | \, r} } e \left( \frac{r\cdot a}{q}\right) \nonumber \\
&=\sum_{t \, | \, q}\boldsymbol{\mu}(t)\sum_{r\in [q/t]^d} e \left( \frac{r\cdot a}{q/t}\right) \nonumber \\
&=\sum_{t \, | \, q}\boldsymbol{\mu}(t) \mathbf{1}_{q/t \, | \, a} \left(\frac{q}{t}\right)^d
= \boldsymbol{\mu}(q) \label{eq: formula for ramanujan sum}
\end{align}
In the second last equality we used orthogonality of the characters on $\Z/(q/t)\Z$ and in the final equality we used that $\gcd(a,q)=1$.
Substituting equations \eqref{eq: little o equation} and \eqref{eq: formula for ramanujan sum} into equation \eqref{eq: Reducing to Ramanujan sum} and taking $N\to \infty$ then gives the desired formula.

It remains to show that $g(\alpha)=0$ when $\ord(\alpha)=\infty$. Use equation \eqref{eq: (4)} again to write
\begin{align}
g_N(\alpha) = \frac{1}{\labs{Q_N^\mathcal{P}}} \sum_{v \in Q_N^\mathcal{P}} e(v\cdot \alpha)&=\frac{1}{\labs{Q_N^\mathcal{P}}} \sum_{n=1}^N \boldsymbol{\mu}(n) \sum_{\substack{v \in Q_N \\ n \, | \, v}}   e(v\cdot \alpha) \nonumber\\
&=\frac{1}{\labs{Q_N^\mathcal{P}}} \sum_{n=1}^N \boldsymbol{\mu}(n) \sum_{v \in Q_{\floor{N/n}}}   e(v\cdot n\alpha). \label{eq: dealing with irrational}
\end{align}
Now for any positive integer $M$ we have that
\[ \sum_{v\in Q_M} e(v\cdot n\alpha) = \sum_{v_1,\ldots,v_d \in [-M,M]} \prod_{i=1}^d e(v_i\cdot n\alpha_i) = \prod_{i=1}^d \sum_{m=-M}^M e(mn \alpha_i).\]
Since $\ord(\alpha)=\infty$ then at least one $\alpha_i$ is irrational, and so by summing a finite geometric series in the above formula we have that
\begin{align*}
\labs{\sum_{v\in Q_M} e(v\cdot n\alpha)} &\leq \labs{2M+1}^{d-1} \labs{\sum_{m=-M}^M e(n\alpha_i)^m}\\
&\leq\labs{2M+1}^{d-1} \labs{\frac{2}{1-e(n\alpha_i)}} = O_\alpha(M^{d-1}),
\end{align*}
and applying this to equation \eqref{eq: dealing with irrational} shows that
\begin{align*}
\labs{g_N(\alpha)} &= \frac{1}{\labs{Q_N^\mathcal{P}}} \labs{\sum_{n=1}^N \boldsymbol{\mu}(n) O_{\alpha}\left(\frac{N^{d-1}}{n^{d-1}} \right)}\\
&\leq \frac{N^{d-1}}{\labs{Q_N^\mathcal{P}}} C_{\alpha} \sum_{n=1}^N \frac{1}{n^{d-1}} = \frac{N^{d-1}}{\labs{Q_N^\mathcal{P}}} O_{\alpha}(\log(N)) \to 0 \text{ as } N \to \infty
\end{align*}
since $\labs{Q_N^\mathcal{P}}$ is on the order of $\labs{Q_N} = (2N+1)^d$ by equation \eqref{eq: Density of primitives}.
\end{proof}

\end{document}